\documentclass[a4paper,11pt, reqno]{amsart}
\usepackage{etex}
\reserveinserts{28}
\usepackage[T1]{fontenc}
\usepackage{lmodern}
\usepackage[utf8]{inputenc}
\usepackage[english]{babel}
\usepackage[stretch=10]{microtype}
\usepackage{mathtools}
\usepackage{amsthm}
\usepackage{amssymb}
\usepackage{url}
\usepackage{enumerate}
\usepackage{bbm}
\usepackage{color,tikz,pgfplots}
\usepackage{pst-func}

\usepackage[left=2.5cm, top=2.5cm,bottom=2.5cm,right=2.5cm]{geometry}

\newtheorem{theorem}{Theorem}[section]

\newtheorem{lemma}[theorem]{Lemma}
\newtheorem{proposition}[theorem]{Proposition}
\newtheorem{corollary}[theorem]{Corollary}

\theoremstyle{definition}

\newtheorem{remark}[theorem]{Remark}
\numberwithin{equation}{section}

\DeclareMathOperator{\Cov}{Cov}
\DeclareMathOperator{\Var}{Var}

\newcommand{\todistr}{\overset{d}{\underset{n\to\infty}\longrightarrow}}
\newcommand{\toas}{\overset{a.s.}{\underset{n\to\infty}\longrightarrow}}

\newcommand{\N}{\ensuremath{{\mathbb N}}}

\newcommand{\B}{\ensuremath{{\mathbb B}}}
\newcommand{\D}{\ensuremath{{\mathbb D}}}

\newcommand{\SSS}{\ensuremath{{\mathbb S}}}
\newcommand{\R}{\ensuremath{{\mathbb R}}}
\newcommand{\E}{\ensuremath{{\mathbb E}}}
\newcommand{\XX}{\ensuremath{{\mathbb X}}}
\newcommand{\YY}{\ensuremath{{\mathbb Y}}}
\newcommand{\Pro}{\ensuremath{{\mathbb P}}}

\newcommand{\vol}{\mathrm{vol}}

\newcommand{\dif}{\,\mathrm d}

\def\dint{\textup{d}}
\newcommand{\bC}{\ensuremath{{\mathbf C}}}

\def\bX{\mathbf{X}}
\def\bY{\mathbf{Y}}

\author[Z. Kabluchko]{Zakhar Kabluchko}
\address{Zakhar Kabluchko: Institut f\"ur Mathematische Stochastik, Westf\"alische Wilhelms-Universit\"at M\"unster, Orl\'eans-Ring 10, 48149 M\"unster, Germany}
\email{zakhar.kabluchko@uni-muenster.de}

\author[J. Prochno]{Joscha Prochno}
\address{Joscha Prochno: School of Mathematics \& Physical Sciences, University of Hull, Cottingham Road,
	Hull, HU6 7RX, United Kingdom}
\email{j.prochno@hull.ac.uk}

\author[C. Th\"ale]{Christoph Th\"ale}
\address{Christoph Th\"ale: Faculty of Mathematics, Ruhr University Bochum, Universit\"atsstra\ss e 150, 44780 Bochum, Germany}
\email{christoph.thaele@rub.de}


\begin{document}

\title[High-dimensional limit theorems for $\ell_p^n$-balls]{High-dimensional limit theorems\\ for random vectors in $\ell_p^n$-balls}

\begin{abstract}
In this paper, we prove a multivariate central limit theorem for $\ell_q$-norms of high-dimensional random vectors that are chosen uniformly at random in an $\ell_p^n$-ball. As a consequence, we provide several applications on the intersections of $\ell_p^n$-balls in the flavor of Schechtman and Schmuckenschl\"ager and obtain a central limit theorem for the length of a projection of an $\ell_p^n$-ball onto a line spanned by a random direction $\theta\in\SSS^{n-1}$. The latter generalizes results obtained for the cube by Paouris, Pivovarov and Zinn and by Kabluchko, Litvak and Zaporozhets. Moreover, we complement our central limit theorems by providing a complete description of the large deviation behavior, which covers fluctuations far beyond the Gaussian scale. In the regime $1\leq p < q$ this displays in speed and rate function deviations of the $q$-norm on an $\ell_p^n$-ball obtained by Schechtman and Zinn, but we obtain explicit constants.
\end{abstract}
\maketitle


\section{Introduction and main results}

Understanding geometric structures in high dimensions has become increasingly important and the last decade has seen a number of breakthrough results, many of a probabilistic flavor, that unfold various unexpected phenomena as well as a certain regularity that occurs in high-dimensional spaces. At the very heart of these discoveries lies the young theory of Asymptotic Geometric Analysis, where deep ideas and methods from analysis, geometry and probability theory meet in a highly non-trivial way. More information on this vivid and promising field can be found, for instance, in the survey articles \cite{G114,G214} and the recent monographs \cite{AsymGeomAnalyBook,GeometryICB}.

Historically, the first high-dimensional central limit theorem was the result of Poincar\'e and Borel, showing that the distribution of the first $k$ coordinates of a uniformly distributed random point in the $n$-dimensional Euclidean ball or on the Euclidean sphere converges to a $k$-dimensional Gaussian distribution as $n\to\infty$, see \cite{DiaconisFreedman} (including a historical discussion) and also \cite{Stam82} for the case that $k$ grows simultaneously with $n$. Arguably one of most prominent results in this direction is Klartag's central limit theorem for convex bodies \cite{KlartagCLT,KlartagCLT2}, a geometric counterpart to the classical central limit theorem. Roughly speaking, it says that most $k$-dimensional marginals of a random vector uniformly distributed in a convex body are approximately Gaussian, provided that $k=k(n) < n^\kappa$ with $\kappa<1/14$. This central limit theorem for convex bodies was already conjectured by Anttila, Ball and Perissinaki in \cite{ABP2003} (for $k=1$), where it has been verified for the case of uniform distributions on convex sets under some additional assumptions. Recent years have seen emerging several other central limit phenomena for various quantities arising in Asymptotic Geometric Analysis. For instance, Paouris, Pivovarov and Zinn \cite{PPZ14} obtained a central limit theorem for the volume of $k$-dimensional random projections ($k$ fixed) of the $n$-dimensional cube as the space dimension $n$ tends to infinity. In the particular case that $k=1$, their result passes over to a central limit theorem for the $\ell_1$-norm $\Vert\theta\Vert_1$ of a point $\theta\in\SSS^{n-1}$ chosen uniformly at random. Independently and by a different method this has also been obtained in \cite[Theorem 3.6]{KLZ15pomi} by Kabluchko, Litvak and Zaporozhets. It is one of the main goals of the present paper to complement the former geometric motivated central limit results and to explore further the Gaussian fluctuations within the latter framework. More precisely, we prove a multivariate central limit theorem for the $\ell_q$-norm of high-dimensional random vectors that are chosen uniformly at random in an $\ell_p^n$-ball $\B_p^n$, thereby deriving a multi-dimensional version of a result of Schechtman and Schmuckenschl\"ager \cite{SS} and Schmuckenschl\"ager \cite{Schmuchenschlaeger2001}. Our central limit theorem is accompanied by several applications on the geometry of $\ell_p^n$-balls in the flavor of Schechtman and Schmuckenschl\"ager. Moreover, for special constellations of the involved parameters $p$ and $q$ we also observe non-central limiting behaviors with exponential or Gumbel limiting distributions.

While central limit theorems underline the universal behavior of Gaussian fluctuations, it is widely known in probability theory that the large deviation behavior, where one considers fluctuations far beyond the scale of the central limit theorem, is much more sensitive to the involved random elements. In the field of Asymptotic Geometric Analysis, the so-called large deviation principle (LDP) only recently entered the stage. The main motivation to study these for random vectors uniformly distributed in convex bodies is to access non-universal features that remain `unseen' when normal fluctuations are considered, and therefore to unveil properties that allow to distinguish between different convex bodies in high-dimensions. In \cite{GantertKimRamanan}, the authors proved an LDP for 1-dimensional projections of random vectors uniformly distributed in the important class of $\ell_p^n$-balls. Their result, in the annealed case, was extended to the Grassmannian setting of higher-dimensional subspaces in \cite{APT16}, where it was proved that the Euclidean norm of the orthogonal projection of a random vector uniformly distributed in $\B_p^n$ onto a random subspace satisfies an LDP. The key ingredient in the proof was a probabilistic representation in the spirit of Schechtman and Zinn \cite{SZ}. Since therein the rotational invariance of the Euclidean ball plays an essential r{\^o}le it is not clear how the result can be extended to consider general $\ell_q$-norms of random vectors instead of the $\ell_2$-norm. In the spirit of this paper, where we do not consider orthogonal projections, but rather the norm of the random vectors itself, we are able to prove such LDPs for $\ell_q$-norms of sequences of random points that are chosen independently and uniformly at random in an $\ell_p$-ball for all possible regimes of $p$ and $q$. This complements on the one hand the study on central and non-central limit theorems in the first part of this paper and at the same time contributes to the recent line of research on the large deviation behavior of geometric quantities that appear in Asymptotic Geometric Analysis. Let us also remark that in the regime $1\leq p < q$ our large deviation principle in Theorem \ref{thm:LDPp<q} displays exactly the behavior known from the deviation result of Schechtman and Zinn \cite[Theorem 3]{SZ} (see also \cite{SZ2} and Naor's work \cite{NaorConeSurface}). In fact, our result is the asymptotic version of that of Schechtman and Zinn. Although neither one implies the other, our result comes with explicit constants and, for deviation parameters $z$ from a fixed interval (compare with Subsection \ref{subsec:schechtman-zinn}), our result is optimal and indeed stronger. 

Let us now present the main theorems of this work. We start with the central and non-central limit theorems and then present the corresponding large deviation counterparts.
For $p\in[1,\infty]$ and $x=(x_1,\ldots,x_n)\in\R^n$, $n\in\N$, we put
$$
\|x\|_p := \begin{cases}
\Big(\sum\limits_{i=1}^n|x_i|^p\Big)^{1/p} &: p<\infty\\
\max\limits_{1\leq i\leq n}|x_i|&: p=\infty\,.
\end{cases}
$$
By $\B_p^n$ we denote the unit ball in $\R^n$ with respect to the $\|\cdot\|_p$-norm, that is,
$$
\B_p^n=\{x\in\R^n:\|x\|_p\leq 1\}\,.
$$
Recall that a random variable is said to be $p$-generalized Gaussian ($1\leq p< \infty$) if it has a density
\begin{align}\label{eq:pGeneralizedGaussianDensity}
f_p(x)=\frac{e^{-|x|^p/p}}{2p^{1/p}\Gamma(1+1/p)}\,,\qquad x\in\R\,,
\end{align}
with respect to the Lebesgue measure on $\R$. We extend this definition to the case $p=\infty$ by letting the $\infty$-generalized Gaussian be a random variable with uniform distribution on $[-1,1]$. We refer to Section \ref{sec:notation and prelim} for any unexplained notion or notation.

\subsection{Main results -- Part A: Central and non-central limit theorems}

Let us introduce the following quantities:
\begin{equation}\label{eq:Mpr}
M_p(r):=\frac{p^{r/p}}{r+1}\,\frac{\Gamma(1+\frac{r+1}{p})}{\Gamma(1+\frac{1}{p})}\qquad(p<\infty)\,,\qquad M_\infty(r):=\frac {1} {r+1}
\end{equation}
and
\[
C_p(r,s):=M_p(r+s) - M_p(r)M_p(s)\,,
\]
where $r,s\geq 0$ and $p\in[1,\infty]$. As we shall see in Lemma \ref{lem:expectation} below, they represent moments and covariances of $p$-generalized Gaussian random variables. We use the convention that $M_\infty(\infty)=0$ and $C_\infty(\infty,\infty)=C_\infty(\infty,q)=0$.

Part (a) of the following result is the multivariate generalization of \cite[Proposition 2.4]{Schmuchenschlaeger2001}. We denote by $\todistr$ convergence in distribution as an involved parameter $n$ tends to infinity.

\begin{theorem}\label{thm:CLTlpBall}
Let $1\leq p\leq\infty$ and $d,n\in\N$. Assume that $1\leq q_1<\ldots<q_d\leq\infty$ and $Z=Z_n$ is a random vector uniformly distributed in $\B_p^n$.
\begin{itemize}
\item[(a)] If $p\neq q_i$ for all $i\in\{1,\ldots,d\}$ and $q_d<\infty$, then we have the multivariate central limit theorem
\[
\sqrt{n}\,\bigg(n^{1/p-1/q_i}\frac{\|Z\|_{q_i}}{M_p(q_i)^{1/q_i}}-1\bigg)_{i=1}^d \todistr N\,,
\]
where $N$ is a centered Gaussian random vector with covariance matrix ${\bf C}=(c_{ij})_{i,j=1}^d$ and
\begin{align*}
c_{ij} &=\frac{C_p(q_i,q_j)}{q_iq_jM_p(q_i)M_p(q_j)}+{C_p(p,p)\over p^2}- \frac{1}{p}\bigg({C_p(q_i,p)\over q_iM_p(q_i)}+{C_p(q_j,p)\over q_jM_p(q_j)}\bigg)\\
&=\begin{cases}
{1\over q_iq_j}\bigg({\Gamma\big({1\over p}\big)\Gamma\big({q_i+q_j+1\over p}\big)\over\Gamma\big({q_i+1\over p}\big)\Gamma\big({q_j+1\over p}\big)}-1\bigg)-{1\over p} &: p<\infty\\
{1\over q_i+q_j+1} &: p=\infty\,.
\end{cases}
\end{align*}
\item[(b)] Assume that $d=1$ and $p=q:=q_1\in[1,\infty]$. Then, we have the non-central limit theorem
\[
n\big(1-\|Z\|_q\big) \todistr E,
\]
where $E$ is an exponential random variable with mean $1$.
\item[(c)] Assume that $d=1$ and $p\in[1,\infty)$. Then, we have the non-central limit theorem
$$
{n^{1/p}\over (p\log n)^{{1\over p}-1}}\|Z\|_\infty -A_n^{(p)}  \todistr G,
$$
where
$$
A_n^{(p)} := p\log n-{1\over p}\Big((1-p)\log(p\log n)+p\log K\Big)\Big)\quad\text{with}\quad K={1\over p^{1/p}\Gamma\big(1+{1\over p}\big)}
$$
and where $G$ is a Gumbel random variable with distribution function $F_G(t)=e^{-e^{-t}}$, $t\in\R$.
\end{itemize}
\end{theorem}
\subsection{Main results -- Part B: Large deviation principles}

We now present the main theorems on the scale of large deviations. The two regimes $p>q$ and $p<q$ need to be treated differently as they exhibit a different behavior from a rate function and speed point of view. We start with the case $q<p$.

\begin{theorem}\label{thm:LDPp>q}
Assume that $1\leq p <\infty$ and let $Z$ be uniformly distributed on $\B_p^n$.
If $1\leq q<p$, the sequence $\|{\bf Z}\|:=(n^{1/p-1/q}\|Z\|_q)_{n\in\N}$ satisfies an LDP with speed $n$ and good rate function
\[
\mathcal{I}_{\bf \|Z\|} (z) = \begin{cases}
\inf\limits_{z=z_1z_2\atop z_1,z_2\geq 0}[ {\mathcal I}_1(z_1)+{\mathcal I}_2(z_2)] &: z\geq 0\\
+\infty &:\text{otherwise}\,.
\end{cases}
\]
Here
$$
{\mathcal I}_1(z) = \begin{cases}
-\log z &: z\in(0,1]\\
+\infty &:\text{otherwise}
\end{cases}
\qquad\text{and}\qquad
\mathcal I_2(z) =
 \begin{cases}
\inf\limits_{x\geq 0,y>0\atop{x^{1/q}y^{-1/p}}=z} \Lambda^*(x,y) &: z\geq 0\\
+\infty &: z<0\,,
\end{cases}
$$
where $\Lambda^*$ is the Legendre-Fenchel transform of the function
\[
\Lambda(t_1,t_2):=\log \int_{\R} e^{t_1|s|^q+t_2|s|^p}\frac{e^{-|s|^p/p}}{2p^{1/p}\Gamma(1+1/p)}\dif s\,.
\]
\end{theorem}

\bigskip

The dual regime where $q>p$ requires different methods, namely large deviations for sums of so-called stretched exponentials. In that case, we have the following result, now with a fully explicit rate function and a speed which is slower than the one of the regime $q<p$.

\begin{theorem}\label{thm:LDPp<q}
Assume that $1\leq p<\infty$ and let $Z$ be uniformly distributed on $\B_p^n$.
If $p<q<\infty$, then the sequence $\|{\bf Z}\|:=(n^{1/p-1/q}\|Z\|_q)_{n\in\N}$ satisfies an LDP with speed $n^{p/q}$ and good rate function
$$
{\mathcal I}_{\bf \|Z\|} (z) = \begin{cases}
{1\over p}\big(z^q-M_p(q)\big)^{p/q} &: z\geq M_p(q)^{1/q}\\
+\infty &: \text{otherwise}\,.
\end{cases}
$$
\end{theorem}

\bigskip

In the regime where $1\leq p<\infty$ the remaining case $p=q$ leads to the following result.

\begin{theorem}\label{thm:LDPp=q}
Assume that $1\leq p<\infty$ and let $Z$ be uniformly distributed on $\B_p^n$. Then the sequence $\|{\bf Z}\|:=(\|Z\|_p)_{n\in\N}$ satisfies an LDP with speed $n$ and good rate function $\mathcal I_1$ defined in Theorem \ref{thm:LDPp>q}.
\end{theorem}

In the special case that $p=\infty$, where the $p$-generalized Gaussian distribution reduces to the uniform distribution on $[-1,1]$, we obtain the following theorem.

\begin{theorem}\label{thm:LDPpInfinity}
Let $Z$ be uniformly distributed on $\B_\infty^n$.
\begin{itemize}
\item[(a)]
If $1\leq q<\infty$, the sequence $\|{\bf Z}\|:=(n^{-1/q}\|Z\|_q)_{n\in\N}$ satisfies an LDP with speed $n$ and good rate function
$$
{\mathcal I}_{\bf \|Z\|} (z) = \begin{cases}
\mathcal{J}^*(z) &: z\geq 0\\
+\infty &: \text{otherwise}\,,
\end{cases}
$$
where $\mathcal{J}^*$ is the Legendre-Fenchel transform of the function $\mathcal{J}(z) = {1\over 2}\int_{-1}^1e^{z|t|^q}\dif t.$
\item[(b)] If $p=q=\infty$, the sequence $\|{\bf Z}\|:=(\|Z\|_\infty)_{n\in\N}$ satisfies an LDP with speed $n$ and good rate function $\mathcal I_1$ defined in Theorem \ref{thm:LDPp>q}.
\end{itemize}
\end{theorem}

The rest of this paper is structured as follows. In Section \ref{sec:Applications}, we discuss a number of applications of our  limit theorems and compare our large deviation results with the work of Schechtman and Zinn \cite{SZ,SZ2} and Naor \cite{NaorConeSurface}. Some notation and background material needed in the proofs of the main results is collected in Section \ref{sec:notation and prelim}, while the two final sections, Sections \ref{sec:ProofCLT} and \ref{sec:ProofLDP}, contain the proofs of Theorems \ref{thm:CLTlpBall} -- \ref{thm:LDPpInfinity}.

Those readers who are familiar with the notions and notation from Asymptotic Geometric Analysis and Probability Theory may directly continue with the next section, Section \ref{sec:Applications}, and go through the applications and comparisons presented there. Others may want to consult Section \ref{sec:notation and prelim} first in which all necessary notation is introduced and some background material is collected.

\section{Applications and Comparisons}\label{sec:Applications}

It is the purpose of this section to present a number of applications of our multivariate central limit theorem pertaining the geometry of $\ell_p$-balls. First, we revisit a result of Schechtman and Schmuckenschl\"ager \cite{SS} (see also \cite{Schmuchenschlaeger2001}), which is based on the $1$-dimensional version of the central limit theorem. This will then be extended to a multivariate set-up. In a similar spirit, we then consider the intersection of two `almost neighboring' $\ell_p^n$-balls, before we present our central limit theorem for the $1$-dimensional projections of $\ell_p^n$-balls, generalizing thereby a result from \cite{KLZ15pomi,PPZ14}. In the last subsection, we compare our large deviation principles with the deviation and concentration results obtained by Schechtman and Zinn \cite{SZ,SZ2} and Naor \cite{NaorConeSurface}. In what follows, we denote by $\vol_n(\cdot)$ the $n$-dimensional Lebesgue measure.

\subsection{Revisiting and extending a result of Schechtman and Schmuckenschl\"ager}

In this subsection we discuss a first and direct consequence of the central limit theorem (part (a) of Theorem \ref{thm:CLTlpBall}) in its $1$-dimensional form. We show that it implies a result of Schechtman and Schmuckenschl\"ager on the volume of the intersection of $\ell_p^n$-balls \cite{SS} and recover another result of Schmuckenschl\"ager \cite{Schmuchenschlaeger2001}. To keep the work as self-contained as possible, we also provide those arguments in the proof of Corollary \ref{cor:schmuckenschlaeger} that are similar to the original ones. In addition, below we shall present an extension of this result.

For $1\leq p,q \leq \infty$ with $p\neq q$, let us define
\[
m_{p,q}:= M_p(q)^{1/q}
\]
(recall the definition of $M_p(r)$ from \eqref{eq:Mpr}) as well as the parameters
\[
c_{p,n}:= n^{1/p} \vol_n(\B_p^n)^{1/n}\qquad\text{and}\qquad c_{q,n}:=n^{1/q}\vol_n(\B_q^n)^{1/n}\,.
\]
Note that, as shown in \cite{SS},
\[
c_{p,n}\to c_p:= 2e^{1/p}p^{1/p}\Gamma\Big(1+\frac{1}{p}\Big)\qquad\text{and}\qquad c_{\infty,n}\to c_\infty:=2
\]
for $p<\infty$, as $n\to\infty$, and similarly for $c_{q,n}$. Let us further define
\begin{align}\label{def:A_pqn and A_pq}
A_{p,q,n}:= \frac{c_{p,n}}{m_{p,q}c_{q,n}}\qquad\text{and}\qquad A_{p,q}:=\lim_{n\to\infty} A_{p,q,n}
\end{align}
and observe that
\begin{align*}
A_{p,q} = \begin{cases}
{\Gamma(1+{1\over p})^{1+{1/q}}\over\Gamma(1+{1\over q})\Gamma({q+1\over p})^{1/q}}\,e^{{1/p}-{1/q}}\,\Big({p\over q}\Big)^{1/q} &: p,q<\infty\\
{1\over\Gamma(1+{1\over q})}\Big({q+1\over qe}\Big)^{1/q} &: p=\infty\text{ and }q<\infty\,,
\end{cases}
\end{align*}
which coincides with the constant in \cite{SS,Schmuchenschlaeger2001}.
The volume-normalized unit balls of $\ell^n_p$ and $\ell^n_q$ shall be denoted by $\D_p^n:=\vol_n(\B_p^n)^{-1/n}\B_p^n$ and $\D_q^n:=\vol_n(\B_q^n)^{-1/n}\B_q^n$, respectively.

\begin{corollary}\label{cor:schmuckenschlaeger}
Let $1\leq p,q\leq\infty$ be such that $q\neq p$ and $q<\infty$. Then, for all $t\geq 0$,
\begin{align}\label{eq:volume intersection}
\vol_n\Big( \D_p^n \cap t\D_q^n\Big)
\rightarrow
\begin{cases}
1 &: A_{p,q}\, t>1 \\
\frac{1}{2} &: A_{p,q}\, t=1 \\
0 &: A_{p,q}\, t<1\,,
\end{cases}
\end{align}
as $n\to\infty$.
\end{corollary}
\begin{proof}
Let $t\geq 0$. For each $n\in\N$ there exists $t_n\geq 0$ such that
\begin{align}\label{eq:choice of t_n}
t_n\cdot \frac{A_{p,q}}{A_{p,q,n}}=t\,.
\end{align}
Since by definition $A_{p,q,n}\to A_{p,q}$, as $n\to\infty$, we see that $t_n\to t$. To treat the case where $A_{p,q}t=1$, we need to analyze the speed of convergence and, more precisely, show that the error tends to zero faster than $\mathcal O({1\over\sqrt{n}})$, as $n\to\infty$. First, we note that by \eqref{eq:volume p-balls} below,
\[
c_{p,n}=n^{1/p}\frac{2\Gamma(1+\frac{1}{p})}{\Gamma(1+\frac{n}{p})^{1/n}}\qquad\text{and}\qquad c_{q,n}=n^{1/q}\frac{2\Gamma(1+\frac{1}{q})}{\Gamma(1+\frac{n}{q})^{1/n}}
\]
and thus,
\[
A_{p,q,n}= \frac{n^{1/p-1/q}}{m_{p,q}} \frac{\Gamma(1+\frac{1}{p})}{\Gamma(1+\frac{1}{q})} \frac{\Gamma(1+\frac{n}{q})^{1/n}}{\Gamma(1+\frac{n}{p})^{1/n}}\,.
\]
Stirling's formula says that, if $z\to\infty$,
\[
\Gamma(z+1)=\sqrt{2\pi z}\,\Big(\frac{z}{e}\Big)^z\Big(1+\mathcal O\Big(\tfrac{1}{z}\Big)\Big)\,.
\]
Therefore, we obtain
\begin{align*}
A_{p,q,n} & = \frac{n^{1/p-1/q}}{m_{p,q}} \frac{\Gamma(1+\frac{1}{p})}{\Gamma(1+\frac{1}{q})} \Big(\frac{p}{q}\Big)^{1/(2n)}\Big(\frac{n}{qe}\Big)^{1/q}\Big(\frac{n}{pe}\Big)^{-1/p}\frac{\Big(1+\mathcal O\Big(\frac{1}{n}\Big)\Big)^{1/n}}{\Big(1+\mathcal O\Big(\frac{1}{n}\Big)\Big)^{1/n}} \cr
& =  {1\over m_{p,q}}\frac{\Gamma(1+\frac{1}{p})}{\Gamma(1+\frac{1}{q})} \Big(\frac{p}{q}\Big)^{1/(2n)}\frac{(pe)^{1/p}}{(qe)^{1/q}}\Big(1+\mathcal O\Big(\tfrac{1}{n}\Big)\Big)^{1/n}\cr
& = {1\over m_{p,q}} \frac{\Gamma(1+\frac{1}{p})}{\Gamma(1+\frac{1}{q})} \Big(\frac{p}{q}\Big)^{1/(2n)}\frac{(pe)^{1/p}}{(qe)^{1/q}}\Big(1+\mathcal O\Big(\tfrac{1}{n^2}\Big)\Big)\,.
\end{align*}
Since $(\frac{p}{q})^{1/(2n)}=1+\mathcal O(\tfrac{1}{n})$, as $n\to\infty$, we get
\[
A_{p,q,n} = {1\over m_{p,q}}\frac{\Gamma(1+\frac{1}{p})}{\Gamma(1+\frac{1}{q})} \frac{(pe)^{1/p}}{(qe)^{1/q}} \Big(1+\mathcal O(\tfrac{1}{n})\Big) = A_{p,q} \Big(1+\mathcal O(\tfrac{1}{n})\Big)\,.
\]
We can now complete the proof of the corollary. First, we observe that, since $Z$ is uniformly distributed in $\B_p^n$,
\begin{align*}
& \Pro\Big(\|Z\|_q \leq t_nA_{p,q}m_{p,q}n^{1/q-1/p}\Big) \\
& = \frac{\vol_n\big(\{z\in \B_p^n \,:\, z\in t_nA_{p,q}m_{p,q}n^{1/q-1/p}\B_q^n \} \big)}{\vol_n(\B_p^n)} \\
& =\vol_n\bigg(\Big\{z\in\vol_n(\B_p^n)^{-1/n}\B_p^n:z\in t_nA_{p,q}m_{p,q}n^{1/q-1/p}\vol_n(\B_p^n)^{-1/n}\B_q^n\Big\}\bigg)\\
&=\vol_n\bigg(\Big\{z\in\D_p^n:z\in t_nA_{p,q}m_{p,q}{c_{q,n}\over c_{p,n}}\D_q^n\Big\}\bigg)\\
&=\vol_n\bigg(\Big\{z\in\D_p^n:z\in{t_nA_{p,q}\over A_{p,q,n}}\D_q^n\Big\}\bigg) \cr
& = \vol_n\Big(\D_p^n\cap t\D_q^n\Big)\,,
\end{align*}
where the last step follows from the choice of $t_n$ (recall the definition from Equation \eqref{eq:choice of t_n}). It is now left to apply our central limit theorem, Theorem \ref{thm:CLTlpBall} (a), with the choice $d=1$. Indeed, we have
\begin{align*}
 \Pro\bigg(\|Z\|_q \leq t_nA_{p,q}m_{p,q}n^{1/q-1/p}\bigg) & = \Pro\bigg(\sqrt{n}\Big(n^{1/p-1/q}m_{p,q}^{-1}\|Z\|_q-1\Big) \leq \sqrt{n}(t_nA_{p,q}-1)\bigg) \\
&
\rightarrow
\begin{cases}
1 &: tA_{p,q}>1 \\
\frac{1}{2} &: tA_{p,q}=1,\\
0 &: tA_{p,q}<1\,.
\end{cases}
\end{align*}
Here, the first and the third follow, because if $tA_{p,q}>1$ or $tA_{p,q}<1$, then, since $t_n\to t$, $\sqrt{n}(t_nA_{p,q}-1)$ converges to $+\infty$ or $-\infty$, respectively. In the case of equality, $tA_{p,q}=1$, since $t_n=t\big(1+\mathcal O(\tfrac{1}{n})\big)$ as shown above,
\[
\sqrt{n}(t_nA_{p,q}-1)=\sqrt{n}\Big(tA_{p,q}(1+\mathcal O(\tfrac{1}{n}))-1\Big)=\sqrt{n}\,\mathcal O(\tfrac{1}{n}) \to 0\,,
\]
as $n\to\infty$. Thus, solely the latter case requires the study of the speed of convergence. The proof is thus complete.
\end{proof}

We shall now discuss an extension of the previous result which shows that in the `critical case' arbitrary limits in the interval $(0,1)$, other than just $\frac{1}{2}$, may occur as well. To this end let $\Phi(\,\cdot\,)$ be the distribution function of a standard Gaussian random variable and denote by $\Phi^{-1}(\,\cdot\,)$ its inverse. Further, recall the definitions of $A_{p,q,n}$ and $A_{p,q}$ from \eqref{def:A_pqn and A_pq}.

\begin{corollary}
Let $1\leq p,q\leq\infty$ be such that $q\neq p$ and $q<\infty$. Further, let $r\in(0,1)$ and for each $n\in\N$ define
$$
t_n:=A_{p,q}^{-1}\bigg({\Phi^{-1}(r)+o(1)\over\sqrt{n}}+1\bigg)\,.
$$
Then, as $n\to\infty$,
$$
\vol_n\Big(\D_p^n\cap t_n{A_{p,q}\over A_{p,q,n}}\D_q^n\Big)\to r\,.
$$
\end{corollary}
\begin{proof}
To prove the claim, consider an arbitrary sequence $(s_n)_{n\in\N}$ of non-negative real numbers. Again, as in the proof of Corollary \ref{cor:schmuckenschlaeger},
\begin{align*}
\vol_n\bigg(\D_p^n\cap s_n{A_{p,q}\over A_{p,q,n}}\D_q^n\bigg) & = \Pro\bigg(\|Z\|_q \leq s_nA_{p,q}m_{p,q}n^{1/q-1/p}\bigg) \cr
& = \Pro\bigg(\sqrt{n}\Big(n^{1/p-1/q}m_{p,q}^{-1}\|Z\|_q-1\Big) \leq \sqrt{n}(s_nA_{p,q}-1)\bigg)\,.
\end{align*}
Say that we want $\sqrt{n}(s_nA_{p,q}-1)$ to converge to some value $s\in\R$, as $n\to\infty$.
To achieve this, we define, for each $n\in\N$,
\[
s_n:= A_{p,q}^{-1}\bigg(\frac{s+o(1)}{\sqrt{n}}+1\bigg).
\]
Since $A_{p,q,n}=A_{p,q}(1+\mathcal O(\frac{1}{n}))$ as we know from the proof of Corollary \ref{cor:schmuckenschlaeger}, this means that
\[
\frac{s_nA_{p,q}}{A_{p,q,n}} = \frac{\frac{s+o(1)}{\sqrt{n}}+1}{A_{p,q,n}} = \frac{\frac{s+o(1)}{\sqrt{n}}+1}{A_{p,q}(1+\mathcal O(\frac{1}{n}))} = \frac{\frac{s+o(1)}{\sqrt{n}}+1}{A_{p,q}}\,.
\]
From this observation and the central limit theorem, we conclude that, as $n\to\infty$,
\begin{align*}
\vol_n\bigg(\D_p^n\cap s_n{A_{p,q}\over A_{p,q,n}}\D_q^n\bigg) &= \vol_n\bigg(\D_p^n\cap A_{p,q}^{-1}\Big(\frac{s+o(1)}{\sqrt{n}}+1\Big)\D_q^n\bigg) \\
&\to \frac{1}{\sqrt{2\pi}}\int_{-\infty}^s e^{-t^2/2}\,\dint t=\Phi(s)\,,
\end{align*}
This proves the result by taking $s=\Phi^{-1}(r)$, in which case $s_n$ coincides with $t_n$.
\end{proof}

\subsection{A multivariate version of the result of Schechtman and Schmuckenschl\"ager}

The purpose of this subsection is to derive a multivariate generalization of the result of Schechtman and Schmuckenschl\"ager discussed in the previous subsection. For that purpose we consider $d\in\N$ volume-normalized $\ell_{q_i}^n$-balls $\D_{q_1}^n,\ldots,\D_{q_d}^n$ and position them relative to another $\ell_p^n$-ball $\D_p^n$. In view of \eqref{eq:volume intersection} one might conjecture that if all these positions are `critical', the volume of the mutual intersection tends to $2^{-d}$. However, our multivariate central limit theorem (Theorem \ref{thm:CLTlpBall} (a)) will show that this is \textit{not} the case. Instead, $2^{-d}$ has to be replaced by the probability that the components of the Gaussian limiting vector from Theorem \ref{thm:CLTlpBall} are negative. If $d=1$, this value is clearly equal to $1/2$, but since the vector is \textit{correlated} such a simple relation cannot be expected to be true in the multivariate set-up where $d>1$. For $d=2$ the probability is explicitly expressed in \eqref{eq:Probabd=2} below.

We shall use the same notation as in the previous subsection and, moreover, denote by $N=(N_1,\ldots,N_d)$ the centered Gaussian random vector with covariance matrix $\bC$ from Theorem \ref{thm:CLTlpBall}. We shall write $\sharp(A)$ for the cardinality of a set $A$.

\begin{corollary}\label{cor:intersection multiple l_p balls}
Fix $d\in\N$. Let $1\leq p\leq\infty$, $1\leq q_1<\ldots<q_d < \infty$ be such that $p\neq q_i$ for all $i\in\{1,\ldots,d\}$, and let $t_1,\ldots,t_d\geq 0$. Define the sets $I_{\star}:=\{i\in\{1,\ldots,d\}:A_{p,q_i}t_i\star 1\}$, where $\star$ is any of the symbols $<$, $=$ or $>$. Then, as $n\to\infty$,
\begin{align*}
&\vol_n\big(\D_p^n\cap t_1\D_{q_1}^n\cap\ldots\cap t_d\D_{q_d}^n\big) \rightarrow\begin{cases}
1 &: \sharp(I_{>})=d\\
\Pro(N_i\leq 0\text{ for all }i\in I_{=}) &: \sharp(I_{=})\geq 1 \text{ and } \sharp(I_{<}) = 0\\
0 &: \sharp(I_{<})\geq 1.
\end{cases}
\end{align*}
\end{corollary}
\begin{proof}
The proof follows along the lines of what has been discussed in the previous subsection. So, for each $i\in\{1,\ldots,d\}$ and $n\in\N$ let $t_n^{(i)}\in\R$ be such that
\begin{align*}
t_n^{(i)}\cdot{A_{p,q_i}\over A_{p,q_i,n}} = t_i\,.
\end{align*}
Since $Z$ is uniformly distributed on $\B_p^n$, we have that
\begin{align*}
&\Pro\big(\|Z\|_{q_1}\leq t_n^{(1)}A_{p,q_1}m_{p,q_1}n^{1/q_1-1/p},\ldots,\|Z\|_{q_d}\leq t_n^{(d)}A_{p,q_d}m_{p,q_d}n^{1/q_d-1/p}\big)\\
&={\vol_n\Big(\Big\{z\in\B_p^n:z\in t_n^{(1)}A_{p,q_1}m_{p,q_1}n^{1/q_1-1/p}\B_{q_1}^n,\ldots,z\in t_n^{(d)}A_{p,q_d}m_{p,q_d}n^{1/q_d-1/p}\B_{q_d}^n\Big\}\Big)\over\vol_n(\B_p^n)}\\
&={\vol_n\Big(\Big\{z\in\B_p^n:z\in t_n^{(1)}{A_{p,q_1}\over A_{p,q_1,n}}\D_{q_1}^n,\ldots,z\in  t_n^{(d)}{A_{p,q_d}\over A_{p,q_d,n}}\D_{q_d}^n\Big\}\Big)\over\vol_n(\B_p^n)}\\
&=\vol_n\big(\D_p^n\cap t_1\D_{q_1}^n\cap\ldots\cap t_d\D_{q_d}^n\big)\,,
\end{align*}
where we used the definitions of $t_n^{(1)},\ldots,t_n^{(d)}$. Moreover, the multivariate central limit theorem, Theorem \ref{thm:CLTlpBall} (a), implies that, as $n\to\infty$,
\begin{align*}
&\Pro\Big(\|Z\|_{q_1}\leq t_n^{(1)}A_{p,q_1}m_{p,q_1}n^{1/q_1-1/p},\ldots,\|Z\|_{q_d}\leq t_n^{(d)}A_{p,q_d}m_{p,q_d}n^{1/q_d-1/p}\Big)\\
&=\Pro\Big(\sqrt{n}(n^{1/p-1/q_1}m_{p,q_1}^{-1}\|Z\|_{q_1}-1)\leq\sqrt{n}(t_n^{(1)}A_{p,q_1}-1),\ldots,\\
&\qquad\qquad\qquad\qquad\qquad\qquad\qquad\ldots,\sqrt{n}(n^{1/p-1/q_d}m_{p,q_d}^{-1}\|Z\|_{q_d}-1)\leq\sqrt{n}(t_n^{(d)}A_{p,q_d}-1)\Big)\\
&\to \begin{cases}
1 &: \sharp(I_{>})=d\\
\Pro(N_i\leq 0\text{ for all }i\in I_{=}) &: \sharp(I_{=})\geq 1 \text{ and } \sharp(I_{<})= 0\\
0 &: \sharp(I_{<})\geq 1.
\end{cases}
\end{align*}
Here, the first and the third case follow since, for each $i\in\{1,\ldots,d\}$, $t_n^{(i)}\to t_i$, as $n\to\infty$, and since $\sqrt{n}(t_n^{(i)}A_{p,q_i}-1)$ converges to $+\infty$ or $-\infty$ depending on whether $t_iA_{p,q_i}>1$ or $t_iA_{p,q_i}<1$. In the equality cases, we can argue coordinate-wise as in the proof of Corollary \ref{cor:schmuckenschlaeger} by analyzing the speed of convergence of $A_{p,q_i,n}$ to $A_{p,q_i}$, which is of order $\mathcal{O}({1\over n})$. This completes the argument.
\end{proof}

\begin{remark}
In statistics, probabilities of the form $\Pro(N_i\leq 0\text{ for all }i\in I_{=})$ are known as quadrant probabilities. For example, if $d=2$ and $I_{=}=\{1,2\}$ the probability $\Pro(N_1\leq 0,N_2\leq 0)$ in the previous theorem can be computed explicitly in terms of the covariances $c_{ij}$, $1\leq i,j\leq 2$, given by Theorem \ref{thm:CLTlpBall}:
\begin{equation}\label{eq:Probabd=2}
\Pro(N_1\leq 0,N_2\leq 0) = {1\over 2\pi}\sqrt{{c_{11}c_{22}\over c_{11}c_{22}-c_{12}}\Big(1-{c_{12}\over c_{11}c_{22}}\Big)}\,\arctan\Big(\sqrt{{c_{11}c_{22}\over c_{12}}-1}\Big)\,.
\end{equation}
\end{remark}

\bigskip

\subsection{Intersection of neighboring $\ell_p^n$-balls}

Let us now compare the volume of the intersection of multiple $\ell_p^n$-balls (similar to Corollary \ref{cor:intersection multiple l_p balls}), when they are approaching a fixed ball as the dimension tends to infinity. More precisely, we consider a multivariate set-up and compare $\B_p^n$ with $\B_{q_1}^n,\ldots,\B_{q_d}^n$, $d\in\N$, where now $q_1=q_1(n),\ldots,q_d=q_d(n)$ depend on $n$ in such a way that
\begin{equation}\label{eq:Defq}
q_1 =p+{\alpha_1+o(1)\over\log n},\ldots,q_d =p+{\alpha_d+o(1)\over\log n}\,,
\end{equation}
where in each case $o(1)$ stands for a sequence tending to zero, as $n\to\infty$, and $\alpha_1,\ldots,\alpha_d\in\R$ are constant such that $q_1,\ldots,q_d\geq 1$ for all $n\geq 2$. In this set-up we obtain the following result in the spirit of Schechtman and Schmuckenschl\"ager (see \cite{SS}) discussed above.

\begin{proposition}\label{prop:neighboring balls}
Fix $d\in\N$. Let $1\leq p < \infty$, $q_1=q_1(n),\ldots,q_d=q_d(n)$ as in \eqref{eq:Defq} and $s_1,\ldots,s_d\geq 0$. Define the set $I_\star:=\{i\in\{1,\ldots,d\}:s_i\star e^{-\alpha_i/p^2}\}$, where $\star$ is either $<$ or $>$. Then, as $n\to\infty$,
$$
{\vol_n(\B_p^n\cap s_1\B_{q_1}^n\cap\ldots\cap s_d\B_{q_d}^n)\over\vol_n(\B_p^n)} \rightarrow\begin{cases}
1 &: \sharp(I_{>})=d\\
0 &: \sharp(I_{<})\geq 1.
\end{cases}
$$
\end{proposition}

\bigskip

We remark that the previous result is not a direct consequence of our central limit theorem rather than its proof. For this reason, the proof of Proposition \ref{prop:neighboring balls} is postponed to Section \ref{sec:ProofCLT}. Moreover, we remark that in contrast to the two previous applications, we are not able to handle the critical case, for example, that $s=e^{-\alpha/p^2}$ in the case that $d=1$.

\subsection{One-dimensional projections of $\ell_q^n$-balls}

As another consequence of our central limit theorem (Theorem \ref{thm:CLTlpBall}), we obtain the following generalization of results of Paouris, Pivovarov and Zinn \cite[page 703]{PPZ14} and Kabluchko, Litvak and Zaporozhets \cite[Theorem 3.6]{KLZ15pomi} who independently of each other obtained the result below in the special case of the $n$-dimensional cube, $\B_\infty^n$. Moreover, the paper \cite{KLZ15pomi} also treats the case of the cross polytope $\B_1^n$, which displays a non-central limit behavior. In what follows, for $\theta\in\SSS^{n-1}$, we shall write $\vol_1(P_\theta\B_q^n)$ for the length of the projection of $\B_q^n$ onto the line spanned by $\theta$. For $1\leq q \leq \infty$, we denote by $q^*$ its conjugate defined via the relation $\frac{1}{q}+\frac{1}{q^*}=1$, and use the convention that $\frac{1}{\infty}=0$.

\begin{corollary}\label{cor:1 dimensional projections p-balls}
Let $1\leq q \leq \infty$, $q\neq 2$, and $\theta\in\SSS^{n-1}$ be chosen at random with respect to the normalized spherical Lebesgue measure on $\SSS^{n-1}$.
\begin{itemize}
\item[(a)] If $q>1$, then
\[
\frac{n^{1/q}\,\vol_1(P_\theta\B_q^n)}{2M_2(q^*)^{1/q^*}}-\sqrt{n} \todistr N\,,
\]
where $N$ is a centered Gaussian random variable with variance
\[
\sigma_q^2:=\frac{1}{(q^*)^2}\bigg(\sqrt{\pi}\,\frac{\Gamma\big(\frac{2q^*+1}{2}\big)}{\Gamma\big(\frac{q^*+1}{2}\big)^2}-1\bigg)-\frac{1}{2}\,.
\]
\item[(b)] If $q=1$, then
$$
\sqrt{2n\log n}\,\vol_1(P_\theta \B_1^n)-2A_n^{(2)}\todistr 2G\,,
$$
where $G$ is Gumbel distributed and $A_n^{(2)}$ is the same as in Theorem \ref{thm:CLTlpBall} (c).
\end{itemize}
\end{corollary}

\begin{remark}
The constant $M_2(q^*)^{1/q*}$ can explicitly be expressed in terms of gamma functions as follows:
$$
M_2(q^*)^{1/q*} = \begin{cases}
\sqrt{2\pi^{\frac {1-q}{q}}}\,\Gamma\Big(\frac{2q-1}{2q-2}\Big)^{1- \frac 1 q} &: q<\infty\\
\sqrt{2\over\pi} &: q=\infty\,.
\end{cases}
$$
\end{remark}

\begin{remark}
We notice that the statement in Corollary \ref{cor:1 dimensional projections p-balls} (b) is consistent with \cite[Theorem 3.7]{KLZ15pomi}, where a slightly different centering than $2A_n^{(2)}$ has been used. However, it can be checked that both sequences are asymptotically equivalent.
\end{remark}

In particular, in the setting of \cite[page 703]{PPZ14} and \cite[Theorem 3.6]{KLZ15pomi} where $q=\infty$ and $q^*=1$, we obtain
\[
M_2(1)=\sqrt{\frac{2}{\pi}}\qquad\text{and}\qquad \sigma_\infty^2=\frac{\pi-3}{2}\,.
\]
Consequently, as $n\to\infty$, we find the central limit theorem
\[
\vol_1(P_\theta\B_\infty^n) - 2\sqrt{\frac{2n}{\pi}} \overset{d}{\longrightarrow} \mathcal N\Big(0,\frac{4\pi-12}{\pi}\Big)\,.
\]
We emphasize here that this is slightly different from the result in \cite{KLZ15pomi}, since here we are working with $[-1,1]^n$, while the central limit theorem in \cite{KLZ15pomi} is formulated for $[-{1\over 2},{1\over 2}]^n$.

\begin{proof}[Proof of Corollary \ref{cor:1 dimensional projections p-balls}]
Note that for any fixed vector $\theta\in\SSS^{n-1}$,
\[
\vol_1(P_\theta\B_q^n) = 2 \sup_{x\in\B_q^n} |\langle x,\theta\rangle|= 2 \|\theta\|_{q^*}\,.
\]
The result in part (a) is now a consequence of Theorem \ref{thm:CLTlpBall} (a) in the form presented in Remark \ref{rem:volume cone measure} in the case $d=1$ if we choose $\theta\in\SSS^{n-1}$ at random with respect to the cone measure on $\SSS^{n-1}$, which coincides in this special case with the normalized surface measure. Part (b) is a consequence of Theorem \ref{thm:CLTlpBall} (c) with the choice $p=2$ there (again in its cone measure version, where the radial part can be omitted).
\end{proof}

\subsection{Comparison with a concentration inequality of Schechtman and Zinn}\label{subsec:schechtman-zinn}

Let us briefly compare the explicit rate function we obtained in the LDP if $p<q$ with the deviation results of Schechtman and Zinn in \cite[Theorem 3]{SZ} (see \cite[Corollary 4]{SZ} for the normalized Lebesgue measure) and of Naor in \cite{NaorConeSurface} (see Theorem 2 there). The authors proved that if $1\leq p < q < \infty$, then
\begin{equation}\label{eq:SZConcentration}
\Pro (n^{1/p-1/q}\|Z\|_q >z) \leq e^{-cn^{p/q}z^p}
\end{equation}
for all $z>T(p,q)$ and with $c=1/T(p,q)$. Here $Z$ can either be uniformly distributed in $\B_p^n$ or distributed according to the cone measure on the boundary of $\B_p^n$. This is in line with the LDP in Theorem \ref{thm:LDPp<q} (see also Remark \ref{rem:LDPCone} for the cone measure case), where also $1\leq p<\infty$ and $p<q$. In this case, we identify $n^{p/q}$ in the exponent on the right hand side of \eqref{eq:SZConcentration} as the speed of the LDP and $z^p$ as the asymptotically leading term of the rate function $\mathcal{I}_{\|{\bf Z}\|}(z)$, as $z\to\infty$. Note that our LDP is in a sense optimal, and we can, contrary to \cite{SZ}, identify the exact constant $1/p$ in the exponent. As already mentioned in the introduction, neither one of the results implies the other. Only for deviation parameters $z$ from a fixed compact interval, our result is optimal and indeed stronger.


\section{Notation and preliminaries}\label{sec:notation and prelim}

We now present the notation and background material that is used throughout the remaining parts of this paper. Since we have a broad readership in mind and aim to keep this work as self-contained as possible, we present the necessary material from probability and, in particular, large deviations theory.

\subsection{General notation}
For a subset $A\subset\XX$ of some topological space $\XX$ we write $A^\circ$ and $\overline{A}$ for the interior and the closure of $A$, respectively.

We shall write $\overset{\text{d}}{\longrightarrow}$ and $\overset{\text{a.s.}}{\longrightarrow}$ to indicate convergence in distribution and almost surely, respectively. Moreover, $X_1\overset{\text{d}}{=}X_2$ indicates that two random elements $X_1$ and $X_2$ have the same distribution.

Given a Borel probability measure $\mu$ on $\R^n$, we shall indicate by $X\sim \mu$ that the random vector $X$ has distribution $\mu$. In particular, we write $X\sim\mathcal{N}(m,\sigma^2)$ if the random variable $X$ has a Gaussian distribution with mean $m\in\R$ and variance $\sigma^2 >0$. Similarly, in the multivariate setting we write $\mathcal{N}({\bf m},\Sigma)$ to indicate the multivariate Gaussian distribution with mean vector ${\bf m}$ and covariance matrix $\Sigma$.

We shall use the standard Landau notation $\mathcal{O}(\cdot)$ and $o(\cdot)$ for sequences as well as for functions and where the asymptotics is considered as the parameters go to $0$ or $\infty$, the precise meaning will always be clear from the context.


\subsection{Probabilistic aspects of $\ell_p^n$-balls}
Recall that for $p\in[1,\infty]$ and $x=(x_1,\ldots,x_n)\in\R^n$, $n\in\N$, we write $\|x\|_p$ for the $p$-norm of $x$ and that we denote the unit ball in $\R^n$ with respect to the $\|\cdot\|_p$-norm by $\B_p^n$. We say that a random variable $X$ has a $p$-generalized Gaussian distribution for some $1\leq p<\infty$ and we write $X\sim G_p$ if $X$ has density $f_p$ given by \eqref{eq:pGeneralizedGaussianDensity} with respect to the Lebesgue measure on $\R$.

We recall from \cite{SZ} the following probabilistic representation for a uniformly distributed random point in $\B_p^n$.

\begin{lemma}[Schechtman and Zinn, \cite{SZ}]\label{lem:SZ}
Fix $1\leq p<\infty$ and let $Z$ be a uniformly distributed random point in $\B_p^n$. Then,
\[
Z\stackrel{d}{=}U^{1/n}\frac{Y}{\|Y\|_p}\,,
\]
where $U$ is uniformly distributed on $[0,1]$ and independent of $Y$, where $Y=(Y_1,\dots,Y_n)$ has independent coordinates $Y_1,\ldots,Y_n\sim G_p$.
\end{lemma}

We shall also exploit the well-known fact that
\begin{align}\label{eq:volume p-balls}
\vol_n(\B_{p}^n) = \frac{\big(2\Gamma(1+\frac{1}{p})\big)^n}{\Gamma(1+\frac{n}{p})}\,.
\end{align}

\subsection{The Skorokhod-Dudley lemma}

We shall use the following technical device that allows to translate convergence in distribution to almost sure convergence on an appropriate probability space.

\begin{lemma}[Skorokhod and Dudley, Theorem 4.30 in \cite{Kallenberg}]\label{lem:Skorohod}
Let $\xi,\xi_1,\xi_2,\ldots$ be random elements taking values in a separable metric space such that $\xi_n\overset{d}{\longrightarrow}\xi$, as $n\to\infty$. Then there exists a probability space with random elements $\widetilde{\xi},\widetilde{\xi}_1,\widetilde{\xi}_2\ldots$ such that $\widetilde{\xi}\overset{d}{=}\xi$, $\widetilde{\xi}_n\overset{d}{=}\xi_n$ for all $n\in\N$ and $\widetilde{\xi}_n\overset{a.s.}{\longrightarrow}\widetilde{\xi}$, as $n\to\infty$.
\end{lemma}

\subsection{Large deviations}

To keep our paper reasonably self-contained, we recall some key concepts from the theory of large deviations. For further background material and references the reader is directed to \cite{DemboZeitouni} and Chapter 27 in \cite{Kallenberg}.

Let us recall that a sequence $\bX:=(X_n)_{n\in\N}$ of random elements taking values in some Hausdorff topological space $\XX$ satisfies a large deviation principle (LDP) with speed $s(n)$ and rate function ${\mathcal I}_\bX$ if $s:\N\to (0,\infty)$, $\mathcal{I}_\bX:\XX\to[0,\infty]$ is lower semi-continuous, and if
\begin{equation*}
\begin{split}
-\inf_{x\in A^\circ}\mathcal{I}_\bX(x) &\leq\liminf_{n\to\infty}{1\over s(n)}\log\Pro(X_n\in A)\\
&\leq\limsup_{n\to\infty}{1\over s(n)}\log\Pro(X_n\in A)\leq-\inf_{x\in\overline{A}}\mathcal{I}_\bX(x)
\end{split}
\end{equation*}
for all Borel sets $A\subset\XX$. One says that the rate function $\mathcal{I}_\bX$ is good if it has compact level sets $\{x\in\XX\,:\, \mathcal{I}_\bX(x) \leq \alpha \}$, $\alpha\in[0,\infty)$.

For a moment, let $\XX=\R^d$ for some $d\in\N$. We denote by $\Lambda^*$ the Legendre-Fenchel transform of a function $\Lambda:\R^d\to\R \cup\{+\infty\}$, which is defined as
$$
\Lambda^*(x):=\sup_{u\in\R^d}[\langle u, x\rangle -\Lambda(u)]\,,\qquad x\in\R^d\,,
$$
where $\langle\,\cdot\,,\,\cdot\,\rangle$ is the standard scalar product on $\R^d$. Moreover, we define the (effective) domain of $\Lambda$ to be the set $D_{\Lambda}:=\{u\in\R^d:\Lambda(u)<\infty\}\subset\R^d$.

\begin{lemma}[Cram\'er's theorem, Theorem 27.5 in \cite{Kallenberg}]\label{lem:cramer}
Let $X,X_1,X_2,\ldots$ be independent and identically distributed random vectors taking values in $\R^d$. Assume that the origin is an interior point of $D_\Lambda$, where $\Lambda(u)=\log \E e^{\langle u,X\rangle}$. Then, the partial sums ${1\over n}\sum\limits_{i=1}^n X_i$, $n\in\N$, satisfy an LDP on $\R^d$ with speed $n$ and good rate function $\Lambda^*$.
\end{lemma}

To treat large deviations for $\ell_p$-balls with $p<q$ we need the following version of Cram\'er's theorem for sums of so-called stretched exponential random variables from \cite{GantertRamananRembart}. We directly formulate it in the form that is needed in our framework.

\begin{lemma}[Cram\'er's theorem for stretched exponentials, Theorem 1 in \cite{GantertRamananRembart}]\label{lem:GantertRamananRembart}
Let $X,X_1,X_2,\ldots$ be non-negative, independent and identically distributed random variables.  
Assume that there are constants $r\in(0,1)$, $t_0>0$ and slowly varying functions $c_1,c_2,b:(0,\infty)\to(0,\infty)$ such that
$$
c_1(t)e^{-b(t)t^r}\leq \Pro(X\geq t)\leq c_2(t)e^{-b(t)t^r}\,,\qquad t\geq t_0\,.
$$
Then, the sequence of random variables ${1\over n}\sum\limits_{i=1}^nX_i$, $n\in\N$, satisfies an LDP on $\R$ with speed $b(n)n^r$ and good rate function
$$
\mathcal{I}(z) = \begin{cases}
(x-\E X)^r &: z\geq \E X\\
+\infty &: \text{otherwise}\,.
\end{cases}
$$
\end{lemma}

\begin{remark}
In \cite{GantertRamananRembart} the result is not formulated as an LDP, but since the random variables are assumed to be non-negative, Theorem 1 in \cite{GantertRamananRembart} can be lifted to an LDP by means of standard methods, see also Remark 3.2 in \cite{GantertRamananRembart}.
\end{remark}

To transform a given LDP to another one by means of a continuous function, we shall use the following version of the contraction principle.

\begin{lemma}[Contraction principle, Theorem 27.11 in \cite{Kallenberg}]\label{lem:contraction principle}
Let $\XX$ and $\YY$ be two Hausdorff topological spaces and $F:\XX\to\YY$ be a continuous function. Further, let $\bX=(X_n)_{n\in\N}$ be a sequence of $\XX$-valued random elements that satisfies an LDP with speed $s(n)$ and good rate function $\mathcal{I}_\bX$. Then the sequence $\bY:=(F(X_n))_{n\in\N}$ satisfies an LDP on $\YY$ with the same speed and with the good rate function $\mathcal{I}_\bY=\mathcal{I}_\bX\circ F^{-1}$, i.e., $\mathcal{I}_\bY(y):=\inf\{\mathcal{I}_\bX(x):F(x)=y\}$, $y\in\YY$, with the convention that $\mathcal{I}_\bY(y)=+\infty$ if $F^{-1}(\{y\})=\emptyset$.
\end{lemma}

Let us also recall that if two sequences of random variables are `exponentially close' they follow the same large deviation behavior

\begin{lemma}[Exponential equivalence, Lemma 27.13 in \cite{Kallenberg}]\label{lem:exponentially equivalent}
Let $\bX=(X_n)_{n\in\N}$ and $\bY=(Y_n)_{n\in\N}$ be two sequences of random variables and assume that $\bX$ satisfies an LDP with speed $s(n)$ and rate function $\mathcal{I}_\bX$. Further, suppose that $\bX$ and $\bY$ are {\rm exponentially equivalent}, i.e.,
$$
\limsup_{n\to\infty}{1\over s(n)}\log \Pro(|X_n-Y_n|>\delta) = -\infty
$$
for any $\delta>0$. Then $\bY$ satisfies an LDP with the same speed and the same rate function as $\bX$.
\end{lemma}

\section{Proof of Theorem \ref{thm:CLTlpBall}}\label{sec:ProofCLT}

We are now prepared to present the proofs of our multivariate central and non-central limit theorems stated in Theorem \ref{thm:CLTlpBall} (a)--(c).
Before, we compute the absolute moments and covariances of $p$-generalized Gaussian random variables.

\begin{lemma}\label{lem:expectation}
Fix $1\leq p\leq\infty$. Let $X\sim G_p$ and $r,s\geq 0$. Then,
$$
\E|X|^r= M_p(r)\qquad\text{and}\qquad \Cov(|X|^r,|X|^s) = C_p(r,s).
$$
In particular, $\Var |X|^r = C_p(r,r)$.
\end{lemma}
\begin{proof}
Let $p<\infty$. We have
\begin{align*}
\E|X|^r & = \frac{1}{2p^{1/p}\Gamma(1+\frac{1}{p})}\int_{-\infty}^{\infty}|x|^re^{-|x|^p/p}\dif x\\
& = \frac{1}{p^{1/p}\Gamma(1+\frac{1}{p})}\int_{0}^{\infty}x^{r-p+1}e^{-x^p/p}\,x^{p-1}\dif x.
\end{align*}
Using the substitution $y=x^p/p$, we obtain
\begin{align*}
\E|X|^r & = \frac{1}{p^{1/p}\Gamma(1+\frac{1}{p})}\int_{0}^{\infty}(yp)^{r-p+1\over p}e^{-y}\,\dif y = {p^{r-p\over p}\over \Gamma(1+{1\over p})}\Gamma\Big({r+1\over p}\Big)\\
&= \frac{p^{r/p}}{r+1}\,\frac{\Gamma(1+\frac{r+1}{p})}{\Gamma(1+\frac{1}{p})}= M_p(r).
\end{align*}
On the other hand, if $p=\infty$, we have
$$
\E|X|^r = {1\over 2}\int_{-1}^1|x|^r\dif x = {1\over r+1} = M_\infty(r).
$$
Moreover, for all $1\leq p\leq\infty$ we have that
\begin{align*}
\Cov(|X|^r,|X|^s) & = \E |X|^{r+s}-\E|X|^r \E|X|^s\\
& =M_p(r+s) - M_p(r)M_p(s) = C_p(r,s)
\end{align*}
and the proof is complete.
\end{proof}

We now present separately the proofs of parts (a), (b), and (c) of Theorem \ref{thm:CLTlpBall}.

\begin{proof}[Proof of Theorem \ref{thm:CLTlpBall}, part (a)]
Assume that $p<\infty$ and fix $i\in\{1,\ldots,d\}$. For any $n\in\N$, define the following random variables
\begin{equation}\label{eq:DefXinEtan}
\xi_n^{(i)} := \frac{1}{\sqrt{n}}\sum\limits_{j=1}^n\Big(|Y_j|^{q_i}-M_p(q_i)\Big)\qquad\text{and}\qquad\eta_n := \frac{1}{\sqrt{n}}\sum\limits_{j=1}^n\Big(|Y_j|^p-1\Big)\,,
\end{equation}
where $Y_1,\ldots,Y_n$ are independent $p$-generalized Gaussians.
Then, by the classical central limit theorem (see, e.g., Proposition 5.9 in \cite{Kallenberg}),
\[
\xi_n^{(i)} \todistr \xi \sim\mathcal N\big(0,C_p(q_i,q_i)\big)
\]
and
\[
\eta_n \todistr \eta \sim\mathcal N\big(0,C_p(p,p)\big)\,.
\]
Moreover, from the multivariate central limit theorem \cite[Theorem 11.10]{Breiman}, we obtain that
\[
\big(\xi_n^{(1)},\ldots,\xi_n^{(d)},\eta_n\big) \todistr \big(\xi^{(1)},\ldots,\xi^{(d)},\eta\big)\sim \mathcal N\big({\bf 0},\Sigma\big)\,,
\]
where ${\bf 0} := (0,\ldots,0)\in\R^{d+1}$ and the covariance matrix is given by
\[
\Sigma =
\begin{pmatrix}
C_p(q_1,q_1) & \ldots & C_p(q_1,q_d) & C_p(q_1,p)   \\
\vdots &  & \vdots & \vdots   \\
C_p(q_d,q_1) & \ldots & C_p(q_d,q_d) & C_p(q_d,p)   \\
\end{pmatrix}.
\]
Using Lemma~\ref{lem:SZ} and the definitions of $\xi_n^{(i)}$ and $\eta_n$, we find that
\begin{align*}
\|Z\|_{q_i} & \stackrel{\text{d}}{=}U^{1/n}\frac{\|Y\|_{q_i}}{\|Y\|_p} \label{eq:ZUY}\\
 \nonumber &= U^{1/n}\frac{\Big(nM_p(q_i)+\sqrt{n}\xi_n^{(i)}\Big)^{1/q_i}}{\Big(n+\sqrt{n}\eta_n\Big)^{1/p}} \\
 \nonumber &\stackrel{\text{}}{=}U^{1/n}\frac{\big(nM_p(q_i)\big)^{1/q_i}}{n^{1/p}}\, \frac{\left(1+\frac{\xi_n^{(i)}}{\sqrt{n}M_p(q_i)}\right)^{1/q_i}}{\left(1+\frac{\eta_n}{\sqrt{n}}\right)^{1/p}}\\
 \nonumber & = U^{1/n} \frac{\big(nM_p(q_i)\big)^{1/q_i}}{n^{1/p}}\, F_i\bigg(\frac{\xi_n^{(i)}}{\sqrt{n}}, \frac{\eta_n}{\sqrt{n}}\bigg)\,,
\end{align*}
where $F_i: \R\times (\R\setminus\{-1 \})\to\R$ is the continuous function given by
\[
F_i(x,y):= \frac{\Big(1+\frac{x}{M_p(q_i)}\Big)^{1/q_i}}{(1+y)^{1/p}}\,.
\]
Using the Skorokhod-Dudley device (Lemma \ref{lem:Skorohod}), we may switch to a probability space carrying random variables $\widetilde \xi_n^{(i)},\widetilde \eta_n, \widetilde \xi^{(i)}$ and $\widetilde \eta$, $n\in\N$, so that
$$
(\widetilde \xi_n^{(1)},\ldots,\widetilde \xi_n^{(d)},\widetilde{\eta}_n)\overset{\text{d}}{=}(\xi_n^{(1)},\ldots, \xi_n^{(d)},{\eta}_n)\qquad\text{and}\qquad(\widetilde \xi^{(1)},\ldots,\widetilde \xi^{(d)},\widetilde{\eta})\overset{\text{d}}{=}(\xi^{(1)},\ldots, \xi^{(d)},{\eta})
$$
and such that for each $i\in\{1,\ldots,d\}$ we have the following almost sure convergence:
\begin{equation}\label{eq:ASConvergence}
\widetilde{\xi}_n^{(i)}\toas\widetilde \xi^{(i)}\quad\text{and}\quad \widetilde{\eta}_n\toas\widetilde{\eta}.
\end{equation}
Again, fix some $i\in\{1,\ldots,d\}$. Note that by~\eqref{eq:ASConvergence} both $\frac{\widetilde \xi_n^{(i)}}{\sqrt{n}}$ and $\frac{\widetilde \eta_n}{\sqrt{n}}$ converge to $0$ almost surely, as $n\to\infty$. Thus, we may use Taylor expansion of $F_i$ around $(0,0)$. We obtain
\begin{align*}
F_i(x,y) &= F_i(0,0)+{\partial F_i\over\partial x}\bigg\vert_{(0,0)}\,x + {\partial F_i\over\partial y}\bigg\vert_{(0,0)}\,y +\mathcal O(x^2+y^2)\\
& = 1+\frac{x}{q_iM_p(q_i)} - {y\over p}+\mathcal O(x^2+y^2)\,.
\end{align*}
Notice that
$$
U^{1/n}=e^{{1\over n}\log(U)}\stackrel{\text{d}}{=}e^{-\frac{E}{n}}\,,
$$
where $E$ is an exponential random variable with mean $1$. Therefore, using the series expansion of the exponential,
\begin{equation}\label{eq:UExp}
e^{-\frac{E}{n}} = 1-\frac{E}{n}+\mathcal O\Big(\frac{1}{n}\Big)\,.
\end{equation}
Let us emphasize that here and below the Landau symbols have to be understood in the almost sure sense. Altogether, we obtain
\begin{align*}
\|Z\|_{q_i} & \stackrel{\text{d}}{=} U^{1/n} \frac{(nM_p(q_i))^{1/q_i}}{n^{1/p}}\, F_i\bigg(\frac{\widetilde \xi_n^{(i)}}{\sqrt{n}}, \frac{\widetilde \eta_n}{\sqrt{n}}\bigg) \\
&=\Big(1-\frac{E}{n}\Big)n^{1/q_i-1/p}M_p(q_i)^{1/q_i}\bigg(1+\frac{1}{q_iM_p(q_i)}\frac{\widetilde \xi_n^{(i)}}{\sqrt{n}}-{1\over p}\frac{\widetilde \eta_n}{\sqrt{n}}+\mathcal O\Big(\frac{1}{n}\Big) \bigg)\\
& =n^{1/q_i-1/p}M_p(q_i)^{1/q_i}\bigg(1+\frac{1}{q_iM_p(q_i)}\frac{\widetilde \xi_n^{(i)}}{\sqrt{n}}-{1\over p}\frac{\widetilde \eta_n}{\sqrt{n}}+\mathcal O\Big(\frac{1}{n}\Big) \bigg)\,.
\end{align*}
Now, it follows from \eqref{eq:ASConvergence} that
\[
\sqrt{n}\,\bigg(\frac{1}{q_iM_p(q_i)}\frac{\widetilde \xi_n^{(i)}}{\sqrt{n}}-{1\over p}\frac{\widetilde \eta_n}{\sqrt{n}}\bigg) \toas \frac{1}{q_iM_p(q_i)}\widetilde \xi^{(i)}-{1\over p}\widetilde \eta\,.
\]
Therefore,  we conclude that
\[
\sqrt{n}\,\bigg(n^{1/p-1/q_i}\frac{1}{M_p(q_i)^{1/q_i}}\|Z\|_{q_i}-1\bigg)\toas \frac{1}{q_iM_p(q_i)}\widetilde \xi^{(i)}-{1\over p}\widetilde \eta\,.
\]
Since $\widetilde \xi^{(i)}$ and $\widetilde \eta$ are jointly Gaussian with covariance matrix $\Sigma_i$ given by
\[
\Sigma_i =
\begin{pmatrix}
C_p(q_i,q_i) & C_p(q_i,p)   \\
C_p(p,q_i) & C_p(p,p)  \\
\end{pmatrix}\,,
\]
we see that the random variable $\frac{1}{q_iM_p(q_i)}\widetilde \xi^{(i)}-{1\over p}\widetilde \eta$ is centered Gaussian with variance
\[
\frac{C_p(q_i,q_i)}{q_i^2M_p(q_i)^2}+{C_p(p,p)\over p^2}- \frac{2C_p(p,q_i)}{pq_iM_p(q_i)}\,.
\]
From the almost sure convergence we also conclude the distributional convergence and on this level we can replace the tilded random variables with the original non-tilded ones. Applying once more the multivariate central limit theorem, we thus conclude that the random vector
$$
\sqrt{n}\,\bigg(n^{1/p-1/q_i}\frac{1}{M_p(q_i)^{1/q_i}}\|Z\|_{q_i}-1\bigg)_{i=1}^d
$$
converges to a centered Gaussian vector with covariance matrix ${\bf C}=(c_{ij})_{i,j=1}^d$, where $c_{ij}$ is given by
\begin{align*}
c_{ij} 
&=\Cov\bigg({1\over q_iM_p(q_i)}\xi^{(i)}-{1\over p}\eta,{1\over q_jM_p(q_j)}\xi^{(j)}-{1\over p}\eta\bigg)\\
&=\frac{\Cov(\xi^{(i)},\xi^{(j)})}{q_iq_jM_p(q_i)M_p(q_j)}+{\Cov(\eta,\eta)\over p^2}- \frac{1}{p}\bigg({\Cov(\xi^{(i)},\eta)\over q_iM_p(q_i)}+{\Cov(\xi^{(j)},\eta)\over q_jM_p(q_j)}\bigg)\\
&=\frac{C_p(q_i,q_j)}{q_iq_jM_p(q_i)M_p(q_j)}+{C_p(p,p)\over p^2}- \frac{1}{p}\bigg({C_p(q_i,p)\over q_iM_p(q_i)}+{C_p(q_j,p)\over q_jM_p(q_j)}\bigg)\,.
\end{align*}
Plugging in \eqref{eq:Mpr} and applying elementary simplifications, we arrive at the explicit representations for $c_{ij}$ in terms of gamma functions.

In the remaining case that $p=\infty$ the above arguments can be repeated by formally putting $M_\infty(\infty):=0$, $C_\infty(\infty,\infty):=0$ and $C_\infty(\infty,q):=0$ for $q\geq 1$.
\end{proof}

\begin{proof}[Proof of Theorem \ref{thm:CLTlpBall}, part (b)]
In the case that $p=q$ and $p<\infty$, the distributional identity in Lemma \ref{lem:SZ} implies that $\|Z\|_q\overset{\text{d}}{=}U^{1/n}$. Therefore, \eqref{eq:UExp} yields the desired result. On the other hand, if $p=q=\infty$, then $\|Z\|_\infty$ has the distribution of the maximum of $n$ independent random variables that are uniformly distributed on $[0,1]$. In this situation, the non-central limit theorem with exponential limiting distribution still holds by a direct computation.
\end{proof}

\begin{proof}[Proof of Theorem \ref{thm:CLTlpBall}, part (c)]
We have the distributional identity
$$
\|Z\|_\infty \overset{d}{=} U^{1/n}{\max(|Y_1|,\ldots,|Y_n|)\over\big(\sum\limits_{i=1}^n|Y_i|^p\big)^{1/p}}\,,
$$
where $U$ is uniformly distributed on $[0,1]$ and, independently of $U$, $Y_1,\ldots,Y_n\sim G_p$ are independent $p$-generalized Gaussian random variables. From the results on p.\ 155 in \cite{EmbrechtsEtAl} it follows that
$$
G_n:={\max(|Y_1|,\ldots,|Y_n|)-d_n\over c_n}\todistr G\,,
$$
where $G$ is Gumbel distributed.
The normalizing constants $c_n$ and $d_n$ can explicitly be chosen as
$$
c_n = (p\log n)^{{1\over p}-1}
$$
and
$$
d_n = (p\log n)^{1\over p}+{1\over p}(p\log n)^{{1\over p}-1}((1-p)\log(p\log n)+p\log K)
$$
with $K=(p^{1/p}\Gamma(1+{1\over p}))^{-1}$ by using the tail asymptotics of the distribution function of $p$-generalized Gaussian random variables.

As in the proof of part (a), the central limit theorem implies that
$$
\eta_n = {\sum\limits_{i=1}^n|Y_i|^p-n\over\sqrt{n}} \todistr \eta\,,
$$
with $\eta\sim\mathcal{N}(0,C_p(p,p))$. The Skorokhod-Dudley device (Lemma \ref{lem:Skorohod}) allows us to switch to a probability space and to random variables $\widetilde{G}_n,\widetilde{G},\widetilde{\eta}_n$ and $\widetilde{\eta}$ such that, for all $n\in\N$, $(\widetilde{G}_n,\widetilde{\eta}_n)\overset{d}{=}(G_n,\eta_n)$, $(\widetilde{G},\widetilde{\eta})\overset{d}{=}(G,\eta)$, $\widetilde{G}_n\overset{a.s.}{\longrightarrow}\widetilde{G}$ and $\widetilde{\eta}_n\overset{a.s.}{\longrightarrow}\eta$, as $n\to\infty$. Again, interpreting Landau symbols in the almost sure sense, we thus conclude that
\begin{align*}
\|Z\|_\infty &\stackrel{\text{d}}{=} U^{1/n}{c_n\widetilde{G}_n+d_n\over(n+\sqrt{n}\,\widetilde{\eta}_n)^{1/p}} = \Big(1+\mathcal{O}\Big({1\over n}\Big)\Big)\,{c_n\widetilde{G}_n+d_n\over n^{1/p}(1+\mathcal{O}(1/\sqrt{n}))} \\
&= \Big(1+\mathcal{O}\Big({1\over \sqrt{n}}\Big)\Big){c_n\widetilde{G}_n+d_n\over n^{1/p}}\,.
\end{align*}
As a consequence, we find
\begin{align*}
{n^{1/p}\|Z\|_\infty-d_n\over c_n} &\stackrel{\text{d}}{=} {(c_n\widetilde{G}_n+d_n)(1+\mathcal{O}(1/\sqrt{n}))-d_n\over c_n}\\
&= {c_n\widetilde{G}_n+d_n+\mathcal{O}(1/\sqrt{n})c_n\widetilde{G}_n+\mathcal{O}(1/\sqrt{n})d_n-d_n\over c_n}\\
&= \widetilde{G}_n+\mathcal{O}\Big({1\over\sqrt{n}}\Big)\,\widetilde{G}_n+\mathcal{O}\Big({1\over\sqrt{n}}\Big)\,{d_n\over c_n}\,.
\end{align*}
Noting that almost surely, $\mathcal{O}\big({1\over\sqrt{n}}\big)\widetilde{G}_n\to 0$ and that by definition of $d_n$ and $c_n$ also $\mathcal{O}\big({1\over\sqrt{n}}\big){d_n\over c_n}\to 0$, as $n\to\infty$, we have thus proved that the right-hand side converges a.s.\ to $\widetilde{G}$.  On the level of distributional convergence, this implies that 
$$
{n^{1/p}\|Z\|_\infty-d_n\over c_n} \todistr \widetilde{G}\,,
$$
which  is precisely the claim of part (c) of Theorem \ref{thm:CLTlpBall}.
\end{proof}

\begin{remark}\label{rem:almost disjoint}
Note that
\begin{align*}
& \Pro(\|Z\|_q\leq 1) = \Pro\bigg(\sqrt{n}\bigg(n^{1/p-1/q}\frac{1}{M_p(q)^{1/q}}\|Z\|_q-1\bigg)\leq \sqrt{n}\Big(n^{1/p-1/q}\frac{1}{M_p(q)^{1/q}}-1\Big) \bigg)\,.
\end{align*}
According to Theorem \ref{thm:CLTlpBall} the random variable on the left hand side satisfies a central limit theorem, while the right hand side tends to $-\infty$, as $n\to\infty$, if $p>q$. Thus,
$$
\lim_{n\to\infty} \Pro(\|Z\|_q\leq 1)  = 0\,.
$$
Roughly speaking this means that in high dimensions, $\mathbb B_p^n \cap \mathbb B_q^n$  is negligible compared to $\mathbb B_p^n$.
\end{remark}

\begin{remark}\label{rem:volume cone measure}
The results of Theorem \ref{thm:CLTlpBall} (a) and (c) continue to hold if $Z$ is chosen on the boundary of $\B_p^n$ with respect to the so-called cone measure $\mu_p$. The latter is defined as
$$
\mu_p(A) := {\vol_n(\{rx:x\in A,r\in[0,1]\})\over\vol_n(\mathbb B_p^n)}\,,
$$
where $A$ is any measurable subset of the boundary of $\B_p^n$. This can be directly seen from the proofs, since the radial part $U^{1/n}$ was negligible in both cases, as $n\to\infty$. For part (b) we notice that if $Z$ lies on the boundary of $\B_p^n$, then trivially $\|Z\|_p=1$.
\end{remark}

\begin{remark}
It is possible to prove  the following result  combining  parts (a) and (c) of Theorem~\ref{thm:CLTlpBall}.  Consider some $1\leq q_1 < \ldots < q_d < \infty$ and let $p\in [1,\infty)$ be such that $p\neq q_i$ for all $i\in \{1,\ldots,d\}$. Then, the $(d+1)$-dimensional random vector
$$
\left( \sqrt{n}\,\bigg(\frac{n^{1/p-1/q_1}\|Z\|_{q_1}}{M_p(q_1)^{1/q_1}}-1\bigg),\ldots,  \sqrt{n}\,\bigg(\frac{n^{1/p-1/q_d}\|Z\|_{q_d}}{M_p(q_d)^{1/q_d}}-1\bigg), 
\frac{n^{1/p} \|Z\|_\infty}{ (p\log n)^{{1\over p}-1}} -A_n^{(p)}\right)
$$
converges as $n\to\infty$ in distribution to $(N,G)$, where $N$ and $G$ are independent and have the same distributions as in (a) and (c).  Indeed, the proof of part (a) is based on the central limit theorem, whereas part (c) is obtained as a consequence of a distributional limit theorem on maxima of random variables. The independence of $N$ and $G$ is a consequence of the well-known fact that the maximum and the empirical mean of an i.i.d.\ sample become  asymptotically independent as the sample size goes to $\infty$ (see, e.g., \cite{Chow79}). We refrain from giving a detailed proof.
\end{remark}

\begin{proof}[Proof of Proposition \ref{prop:neighboring balls}]
For $t_1,\ldots,t_d>0$ we have that
$$
{\vol_n\big(\B_p^n\cap t_1\B_{q_1}^n\cap\ldots\cap t_d\B_{q_d}^n\big)\over\vol_n(\B_p^n)} = \Pro(\|Z\|_{q_1}\leq t_1,\ldots,\|Z\|_{q_d}\leq t_d)\,,
$$
where $Z$ is chosen uniformly at random in $\B_p^n$. Recalling from \eqref{eq:DefXinEtan} the definitions of $\xi_n^{(i)}$ and $\eta_n$ as well as their properties shown in the proof of Theorem \ref{thm:CLTlpBall}, we conclude from Lemma \ref{lem:SZ} that
$$
\bigg(\|Z\|_{q_i}\bigg)_{i=1}^d \overset{\text{d}}{=} \bigg( U^{1/n}{(nM_p(q_i))^{1/q_i}\over n^{1/p}}{(1+{\xi_n^{(i)}\over \sqrt{n} M_p(q_i)})^{1/q_i}\over(1+{\eta_n\over \sqrt{n}})^{1/p}}\bigg)_{i=1}^d\,.
$$
The first and the third term converge in probability and hence in distribution to $1$. Moreover, $
{M_p(q_i)^{1/q_i}}\rightarrow 1$, as $n\to\infty$. Finally, using the choice \eqref{eq:Defq} of $q_i=q_i(n)$, we find that
\begin{align}\label{eq:NearByEq1}
{n^{1/q_i}\over n^{1/p}} = e^{({1\over q_i}-{1\over p})\log n} = e^{-{\alpha_i+o(1)\over p^2\log n+\alpha_i p+o(1)}\log n}= e^{-{\alpha_i+o(1)\over (p^2+o(1))\log n}\,\log n} \to e^{-\alpha_i/p^2}\,,
\end{align}
as $n\to\infty$. We have thus proved the claim.
\end{proof}

\section{A Large Deviation Principle for the $q$-norm}\label{sec:ProofLDP}

We now prove a large deviation counterpart to the central limit theorem. We need to consider two cases, $p<q$ and $p>q$, separately. The proofs follow the general ideas of \cite{APT16} and \cite{GantertKimRamanan}. Again, let $Z$ be a point chosen uniformly at random in $\B_p^n$. We will show a large deviations principle for the sequence of random variables
\[
(n^{1/p-1/q}\|Z\|_q)_{n\in\N}\,.
\]

\subsection{The regime $p>q$}

Let us start with the case that $1\leq p<\infty$ and $p>q$, the special case that $p=\infty$ shall be treated separately below.

\begin{proof}[Proof of Theorem \ref{thm:LDPp>q}]
First note that for each $n\in\N$,
\[
n^{1/p-1/q}\|Z\|_q\stackrel{\text{d}}{=}n^{1/p-1/q}U^{1/n}\frac{\|Y\|_q}{\|Y\|_p}
\]
according to Lemma \ref{lem:SZ}, where $Y=(Y_1,\ldots,Y_n)$ is a vector of independent $p$-generalized Gaussian random variables and $U$ is an independent random variable uniformly distributed on $[0,1]$. Let us define
\[
W_n := n^{1/p-1/q}\frac{\|Y\|_q}{\|Y\|_p}\,,\qquad n\in\N\,.
\]
For $n\in\N$, we further define the random vector
\[
S_n:= \frac{1}{n}\sum_{i=1}^n \big(|Y_i|^q,|Y_i|^p\big)\,.
\]
Then, for $t=(t_1,t_2)\in\R^2$,
\[
\Lambda(t_1,t_2)= \log \E e^{\langle t, (|Y_1|^q,|Y_1|^p) \rangle} = \log \int_0^\infty e^{t_1s^q+(t_2-1/p)s^p}\frac{\dif s}{2p^{1/p}\Gamma(1+1/p)}
\]
is the log-moment generating function of $S_n$. Its effective domain is $\R\times(-\infty,1/p)$ if $q<p$ and $\{(t_1,t_2)\in\R^2\,:\,t_1+t_2<1/p\}$ for $q=p$. Therefore, by Cram\'er's theorem (Lemma \ref{lem:cramer}), the sequence ${\bf S}:=(S_n)_{n\in\N}$ satisfies an LDP in $\R^2$ with speed $n$ and good rate function $\Lambda^*$, the Legendre-Fenchel transform of $\Lambda$. One can check that $\Lambda^*(t_1,t_2)= + \infty$ if $t_1\leq 0$ or $t_2\leq 0$. This implies that the sequence ${\bf S}$ also satisfies an LDP on $[0,\infty)\times(0,\infty)$ with the same good rate function $\Lambda^*$.

Next, we define the continuous function
\[
F:[0,\infty)\times(0,\infty)\to\R,\quad (x,y)\mapsto x^{1/q}y^{-1/p}\,.
\]
Note that, for each $n\in\N$, $W_n\overset{d}{=}F(S_n)$. Thus, by the contraction principle (Lemma \ref{lem:contraction principle}), the random sequence ${\bf W}:=(W_n)_{n}$ satisfies an LDP on $\R$ with speed $n$ and good rate function
\[
\mathcal I_{\bf W}(z) =
 \begin{cases}
\inf\limits_{x\geq 0,y>0\atop{x^{1/q}y^{-1/p}}=z}\Lambda^*(x,y) &: z\geq 0\\
+\infty &: z<0\,.
\end{cases}
\]
Finally, let us define the random variables $V_n:=(U^{1/n},W_n)$ and recall from Lemma 3.3 in \cite{GantertKimRamanan} that the sequence ${\bf U}:=(U^{1/n})_{n\in\N}$ satisfies an LDP on $\R$ with speed $n$ and rate function
$$
{\mathcal I}_{\bf U}(z) = \begin{cases}
-\log z &: z\in(0,1]\\
+\infty &:\text{otherwise}\,.
\end{cases}
$$
Since $U^{1/n}$ and $W_n$ are independent, the sequence ${\bf V}:=(V_n)_{n\in\N}$ satisfies an LDP on $\R^2$ with speed $n$ and good rate function
$$
{\mathcal I}_{\bf V} (z_1,z_2) := {\mathcal I}_{\bf U}(z_1)+{\mathcal I}_{\bf W}(z_2)\,,\qquad (z_1,z_2)\in\R^2\,,
$$
see Lemma 2.2 in \cite{Arcones2002} or Proposition 2.6 in \cite{APT16}. Applying once more the contraction principle (Lemma \ref{lem:contraction principle}), this time to the continuous function
$$
F:\R^2\to\R,\quad (x,y)\mapsto xy\,,
$$
we conclude that the sequence of random variables $U^{1/n}W_n\overset{d}{=}n^{1/p-1/q}\|Z\|_q$ ($n\in\N$) satisfies an LDP on $\R$ with speed $n$ and good rate function
$$
\mathcal{I}_{\bf \|Z\|} (z) = \inf_{z=z_1z_2}{\mathcal I}_{\bf V} (z_1,z_2) = \begin{cases}
\inf\limits_{z=z_1z_2\atop z_1,z_2\geq 0}{\mathcal I}_{\bf V} (z_1,z_2) &: z\geq 0\\
+\infty &:\text{otherwise}\,.
\end{cases}
$$
This completes the argument.
\end{proof}

\begin{remark}
We would like to remark that if instead of the uniform distribution the cone measure on $\B_p^n$ is considered (recall the definition in Remark \ref{rem:volume cone measure}), probabilities of the type $\Pro(n^{1/p-1/q}\|Z\|_q\geq x)$ for sufficiently large $x$ were already considered in \cite{ShaoSelfNormalizedLargeDeviations} (see also \cite[Chapter~3]{delaPena_etal_book}) in the context of self-normalized large deviations. The result in \cite{ShaoSelfNormalizedLargeDeviations} applies to far more general situations and its proof is in fact highly technical, while we are relying on elementary principles from large deviation theory.
\end{remark}

\subsection{The regime $p<q$}

After having investigated the case $p>q$ we now turn to the situation where $1\leq p<q<\infty$. As a matter of fact, this case is more delicate because of the lack of finite exponential moments. However, instead of Cram\'er's theorem we now use a large deviation result for sums of so-called stretched exponential random variables in Lemma \ref{lem:GantertRamananRembart}.

\begin{proof}[Proof of Theorem \ref{thm:LDPp<q}]
For each $n\in\N$, let $Y_1,\ldots,Y_n\sim G_p$ be independent $p$-generalized Gaussian random variables and define
$$
S_n:=\sum_{i=1}^n|Y_i|^q\,.
$$
We are going to apply Lemma \ref{lem:GantertRamananRembart} to prove an LDP for the random sequence ${\bf S}:=(S_n)_{n\in\N}$. Using the well-known fact that
$$
{z\over z^p+1}e^{-z^p/z} \leq \int_z^\infty e^{-t^p/p}\dif t\leq{1\over z^{p-1}}e^{-z^p/p}\,,\qquad z\geq 0\,,
$$
one easily verifies that, for sufficiently large $z$,
$$
c_1(z) e^{-b(z)z^{p/q}} \leq \Pro(|Y_1|^q\geq z)\leq c_2(z) e^{-b(z)z^{p/q}}
$$
with suitable slowly varying functions $c_1,c_2,b:(0,\infty)\to(0,\infty)$, where $b(z)$ satisfies $b(z)\to{1\over p}$, as $z\to\infty$. Thus, Lemma \ref{lem:GantertRamananRembart} implies that $\bf S$ satisfies an LDP on $\R$ with speed $n^{p/q}$ and good rate function
$$
{\mathcal I}_{\bf S}(z) = \begin{cases}
{1\over p}(z-M_p(q))^{p/q} &: z\geq M_p(q)\\
+\infty &:\text{otherwise}\,.
\end{cases}
$$
Moreover, applying the contraction principle (Lemma \ref{lem:contraction principle}) with the function $F(z)=z^{1/q}$, $z\in\R$, we conclude that the sequence $(F(S_n))_{n\in\N}=((S_n)^{1/q})_{n\in\N}$ satisfies an LDP on $\R$ with speed $n^{p/q}$ and good rate function ${\mathcal I}_{\bf \|Z\|} (z)$ from the statement of the theorem.

Recalling Lemma \ref{lem:SZ} we have the distributional identity
$$
n^{1/p-1/q}\|Z\|_q \overset{d}{=}U^{1/n}{\Big({1\over n}\sum\limits_{i=1}^n|Y_i|^q\Big)^{1/q}\over\Big({1\over n}\sum\limits_{i=1}^n|Y_i|^p\Big)^{1/p}}\,.
$$
Now, let us fix some $\delta>0$ and $\varepsilon\in(0,1)$. We observe that
\begin{align*}
&\Pro\left(\,\left|\Big({1\over n}\sum\limits_{i=1}^n|Y_i|^q\Big)^{1/q}-U^{1/n}{\Big({1\over n}\sum\limits_{i=1}^n|Y_i|^q\Big)^{1/q}\over\Big({1\over n}\sum\limits_{i=1}^n|Y_i|^p\Big)^{1/p}}\right|>\delta\right)\\
&=\Pro\left(\,\Big({1\over n}\sum\limits_{i=1}^n|Y_i|^q\Big)^{1/q}\left|1-{U^{1/n}\over\Big({1\over n}\sum\limits_{i=1}^n|Y_i|^p\Big)^{1/p}}\right|>\delta\right)\\
&\leq \Pro\bigg(\Big({1\over n}\sum\limits_{i=1}^n|Y_i|^q\Big)^{1/q}>{\delta\over\varepsilon}\bigg)+\Pro\left(1-{U^{1/n}\over\Big({1\over n}\sum\limits_{i=1}^n|Y_i|^p\Big)^{1/p}}>\varepsilon\right)+\Pro\left(1-{U^{1/n}\over\Big({1\over n}\sum\limits_{i=1}^n|Y_i|^p\Big)^{1/p}}<-\varepsilon\right)\\
&\leq \Pro\bigg(\Big({1\over n}\sum\limits_{i=1}^n|Y_i|^q\Big)^{1/q}>{\delta\over\varepsilon}\bigg) + \Pro\big(U^{1/n}<\sqrt{1-\varepsilon}\,\big)+\Pro\bigg({1\over n}\sum\limits_{i=1}^n|Y_i|^p>(1-\varepsilon)^{-p/2}\bigg)\\
&\hspace{4.6cm}+\Pro\big(U^{1/n}>\sqrt{1+\varepsilon}\,\big)+\Pro\bigg({1\over n}\sum\limits_{i=1}^n|Y_i|^p<(1+\varepsilon)^{-p/2}\bigg)\,.
\end{align*}
By Cram\'er's theorem (Lemma \ref{lem:cramer}), the four last terms decay exponentially with speed $n$ (in fact, the rate functions in the corresponding LDPs do not vanish in $U\setminus\{1\}$, where $U\subset\R$ is an open neighborhood of $1$), while the exponential asymptotic of the first term has already been determined above. As a consequence and since $p/q<1$, we have that
\begin{align*}
&\limsup_{n\to\infty}{1\over n^{p/q}}\log\Pro\left(\,\left|\Big({1\over n}\sum\limits_{i=1}^n|Y_i|^q\Big)^{1/q}-U^{1/n}{\Big({1\over n}\sum\limits_{i=1}^n|Y_i|^q\Big)^{1/q}\over\Big({1\over n}\sum\limits_{i=1}^n|Y_i|^p\Big)^{1/p}}\right|>\delta\right)\\
&\qquad\qquad\leq\limsup_{n\to\infty}{1\over n^{p/q}}\log\Pro\bigg(\Big({1\over n}\sum\limits_{i=1}^n|Y_i|^q\Big)^{1/q}>{\delta\over\varepsilon}\bigg)\,,
\end{align*}
which is $-{1\over p}(({\delta\over\varepsilon})^q-M_p(q))^{p/q}$ if $\delta/\varepsilon\geq M_p(q)^{1/q}$ and $-\infty$ otherwise. Letting $\varepsilon\to 0$, we conclude that the above limit is equal to $-\infty$. Thus, the two sequences $(n^{1/p-1/q}\|Z\|_q)_{n\in\N}$ and $((S_n)^{1/q})_{n\in\N}$ are exponentially equivalent and hence obey the same LDP according to Lemma \ref{lem:exponentially equivalent}. This completes the proof of Theorem \ref{thm:LDPp<q}.
\end{proof}

\begin{remark}\label{rem:LDPCone}
The result of Theorem \ref{thm:LDPp<q} (a) continues to hold if $Z$ is chosen on the boundary of $\B_p^n$ with respect to the cone probability measure (recall Remark \ref{rem:volume cone measure}). This can directly be seen from the proof, since the radial part $U^{1/n}$ is negligible, as $n\to\infty$. In this situation we remark that the LDP could also be concluded from the Sanov-type result in \cite{KimRamanan} via the contraction principle.
\end{remark}


\subsection{The regime $p=\infty$}

We consider now the special case that $p=\infty$.

\begin{proof}[Proof of Theorem \ref{thm:LDPpInfinity}, part (a)]
For $n\in\N$ let $Y_1,\ldots,Y_n$ be independent and uniformly distributed on $[-1,1]$. Then,
$$
n^{1/q}\|Z\|_q \overset{d}{=} \Big({1\over n}\sum_{i=1}^n|Y_i|^q\Big)^{1/q}\,.
$$
By Cam\'er's theorem, the sequence $\bf S$ of random variables $S_n={1\over n}\sum_{i=1}^n|Y_i|^q$ satisfies an LDP on $\R$ with speed $n$ and good rate function $\mathcal{I}_{\bf S}^*$, the Legendre-Fenchel transform of the function
$$
\mathcal{I}_{\bf S}(z) = {1\over 2}\int_{-1}^1e^{z|t|^q}\dif t\,.
$$
Finally, applying the contraction principle (Lemma \ref{lem:contraction principle}) to the continuous function $F(x)=x^{1/q}$, $x\in\R$, yields that $(n^{1/q}\|Z\|_q)_{n\in\N}$ satisfies an LDP on $\R$ with speed $n$ and good rate function
$$
\inf_{F(x)=z}\mathcal{I}_{\bf S}^*(x) = {\mathcal I}_{\bf \|Z\|} (z) \,.
$$
The proof is thus complete.
\end{proof}

\begin{proof}[Proof of Theorem \ref{thm:LDPpInfinity}, part (b)]
If $Z$ is uniformly distributed in $\B_\infty^n$, then $\|Z\|_\infty$ is just the maximum of $n$ independent and uniformly distributed random variables. This maximum has probability density $t\mapsto n(1-t)^{n-1}$, $t\in[0,1]$ and thus coincides in distribution with $U^{1/n}$ with $U$ uniformly distributed on $[0,1]$. The result thus follows from Lemma 3.3 in \cite{GantertKimRamanan}.
\end{proof}

\bibliographystyle{plain}
\bibliography{clt}
\end{document}